\documentclass[12pt,twoside]{amsart}
\usepackage{}
\usepackage{amsmath, amsthm, amscd, amsfonts, amssymb, graphicx, color}
\usepackage[bookmarksnumbered, plainpages]{hyperref}

\newtheorem{thm}{Theorem}[section]
\newtheorem{cor}[thm]{Corollary}
\newtheorem{lem}[thm]{Lemma}
\newtheorem{prop}[thm]{Proposition}

\theoremstyle{remark}
\newtheorem{remark}{Remark}[section]
\theoremstyle{remark}
\newtheorem{example}{Example}[section]
\numberwithin{equation}{section}

\def\3{$\mathbf{P}^{\frac{3+p}{2}}(\mathbf{C})$}

\begin{document}

\title{\bf Minimality of a Kind of Pseudo-Umbilical Totally Real Submanifolds in Non-Flat Complex Space Forms}
\author{Liang Zhang*, Pan Zhang}

\address{School of Mathematics and Computer Science, Anhui Normal University\\
Anhui 241000, P.R. China
}
\email{zhliang43@163.com}

\address{School of Mathematical Sciences\\
University of Science and Technology of China\\
Anhui 230026, P.R. China\\
}

\email{panzhang@mail.ustc.edu.cn}

\thanks{{\scriptsize
\hskip -0.4 true cm
\newline \hskip -0.4 true cm \textit{2010 Mathematics Subject Classification.} 53C42; 53C40
\newline *Corresponding author.
}}

\maketitle

\begin{abstract}
 In this paper, by studying the position of umbilical normal vectors in the normal bundle, we prove that pseudo-umbilical totally real submanifolds with flat normal connection in non-flat complex space forms must be minimal.
\end{abstract}

\vskip 0.2 true cm


\pagestyle{myheadings}
\markboth{\rightline {\scriptsize L. ZHANG AND P. ZHANG}}
         {\leftline{\scriptsize Minimality of a Kind of Pseudo-Umbilical Totally Real Submanifolds }}

\bigskip
\bigskip


\section{\bf Introduction}
\vskip 0.4 true cm

Let $\tilde {M}^{n+p}(c)$ be a complex $(n+p)$-dimensional complex space form endowed with the Fubini-Study metric of constant holomorphic curvature $c$. An $n$-dimensional submanifold $M^n$ in $\tilde {M}^{n+p}(c)$ is called totally real if the complex structure $J$ of $\tilde {M}^{n+p}(c)$ carries each tangent space of $M^n$ into its corresponding normal space. Specially, $M^n$ is called Lagrangian if $p=0$. This kind of submanifolds appear naturally in the context of classical mechanics and mathematical physics and have been studied by many geometers. For instance, Chen \cite{C1,C2,C3,C4,C5,C6} have classified Lagrangian surfaces of constant curvature in complex space forms. Shu \cite{S1} proved some integral
inequalities of Simons' type for $n$-dimensional compact Extremal Lagrangian submanifolds in complex space forms and gave some rigidity and characterization theorems for general complex co-dimension $p$. Chen-Ogiue \cite{C7} first studied minimal totally real submanifolds. Yano-Kon \cite{YK1,YK2} studied totally real submanifolds satisfying certain conditions on the second fundamental form and provided some basic examples. After that, as a generalization of the minimal ones, totally real submanifolds with parallel mean curvature vector field have been studied (see, for example, \cite{C8}, \cite{U1}, \cite{U2}).

As we all know, pseudo-umbilical submanifolds can be viewed as
another natural generalization of the minimal case. From Chen's classification result of slumbilical submanifolds in complex space forms (see \cite{C9}), one can construct examples of pseudo-umbilical totally real submanifolds in $\tilde {M}^{n+p}(c)$. However, the class of such submanifolds is too wide to classify, so it is reasonable to study pseudo-umbilical totally real submanifolds under some additional conditions. For example, the author \cite{Z1} proved that complete pseudo-umbilical Lagrangian surfaces in a complex projective plane must be minimal, and \cite{Z2} proved that a pseudo-umbilical totally real submanifold $M^n$ with flat normal connection in a complex projective space $\mathbb{C}P^{n+p}$ is minimal if $n=2$ or $p=0$. The main purpose of this paper is to show that the latter result also holds for general $n$ and general $p$, and the ambient space can be assumed to be any non-flat complex space forms. We state it as the following theorem.

\begin{thm}
  Let $M^n$ be an $n$-dimensional pseudo-umbilical totally real submanifold in a complex space form $\tilde {M}^{n+p}(c)\ (c\not=0)$. If the normal connection is flat, then $M^n$ must be minimal.
\end{thm}

\begin{remark}
  To prove the above theorem, we characterize the umbilical normal vector by its position in the normal bundle in Section 3 (see Proposition 3.1).
This characterization, on the one hand, can provide suitable moving frames for us in the proof of the main theorem. On the other hand, it also implies that there exist
no totally umbilical totally real submanifolds whose normal connections are flat in non-flat complex space forms (see Corollary 3.2).
\end{remark}

\section{\bf Preliminaries}
\vskip 0.4 true cm

Let $M^n$ be an $n$-dimensional totally real submanifold in $\tilde {M}^{n+p}(c)$. Choose a local field of orthonormal frames
\begin{equation}
\begin{split}
& e_1,\cdots,e_n,e_{n+1},\cdots,e_{n+p},\\
e_{1^*}=Je_1,\cdots,e_{n^*}= & Je_n,e_{(n+1)^*}=Je_{n+1},\cdots,e_{(n+p)^*}=Je_{n+p}
\end{split}
\end{equation}
in $\tilde {M}^{n+p}(c)$ in such a way that, restricted to $M^n$,
$e_1,\cdots,e_n$ are tangent to $M^n$. For convenience, we use the
following convention on the range of indices:
\begin{equation*}
\begin{split}
  A,B,C,\cdots & = 1,\cdots,n+p,\,1^*,\cdots,(n+p)^*;\\
  i,j,k,\cdots & = 1,\cdots,n;\\
  \alpha,\beta,\gamma,\cdots & = n+1,\cdots,n+p,\,1^*,\cdots,(n+p)^*;\\
  \lambda,\mu,\cdots & =  n+1,\cdots,n+p.
\end{split}
\end{equation*}
With respect to the frame field of $\tilde {M}^{n+p}(c)$ chosen above,
$J$ has the component[8]
\begin{equation}
(J_{AB})=
{\begin{array}{c@{\hspace{-12pt}}l}
   \left(\begin{array}{c|c}
   \large\textbf{0} & -I_{n+p}\\ [10pt]
   \hline
   \raisebox{-3pt}[0pt]{$\ \ I_{n+p}$} &
   \raisebox{-3pt}[0pt]{\large\textbf{0}}\\ [10pt]
\end{array}\right)&
\begin{array}{l}\left.\rule{0pt}{2pt}\right\}{\mbox{\scriptsize $i$}}\\ [-1pt]
   \left.\rule{0pt}{2pt}\right\}{\mbox{\scriptsize $ \lambda
   $}}\\ [-1pt]
   \left.\rule{0pt}{2pt}\right\}{\mbox{\scriptsize{$i^*$}}}\\ [-1pt]
   \left.\rule{0pt}{2pt}\right\}{\mbox{\scriptsize{{$\lambda^*$}}}}
   \end{array}\\ [-11pt]
\begin{array}{cccc}
   \underbrace{\rule{0pt}{0pt}}_i & \hspace{-5pt}
   \underbrace{\rule{0pt}{0pt}}_{\lambda} &
   \hspace{-5pt}
   \underbrace{\rule{0pt}{0pt}}_{i^*} &
   \hspace{-5pt}
   \underbrace{\rule{0pt}{0pt}}_{\lambda^*}
   \end{array} &
\end{array}}
\end{equation}
where $I_{n+p}$ denotes the identity matrix of degree $n+p$. Let$\{\omega^A\}$
be the dual frames of $\{e_A\}$, then the structure equations of
$\tilde {M}^{n+p}(c)$ are given by
\begin{equation}
\mathrm{d}\omega^A=\sum_B \omega^B\wedge\omega^A_B,
\end{equation}
\begin{equation}
\mathrm{d}\omega^A_B=\sum_C\omega^C_B\wedge\omega^A_C+\frac{1}{2}\sum_{C,D}K_{ABCD}\,
\omega^C\wedge\omega^D,
\end{equation}
where [8]
\begin{equation}
\begin{split}
\omega^j_i =\omega^{j^*}_{i^*},\,\,\omega^{\lambda}_{i} & = \omega^{\lambda
^*}_{i^*},\,\,\omega^{j^*}_i=\omega^{i^*}_j,\\
\omega^{\lambda^*}_i=\omega^{i^*}_{\lambda},\,\,\omega^{\mu}_{\lambda} & = \omega^{\mu
^*}_{\lambda ^*},\,\,\omega^{\mu ^*}_{\lambda}=\omega^{\lambda^*}_{\mu};
\end{split}
\end{equation}
\begin{equation}
K_{ABCD}=\frac{c}{4}(\delta _{AC} \delta _{BD}-\delta_{AD}\delta_{BC}+J_{AC} J_{BD}-J_{AD} J_{BC}+
2J_{AB} J_{CD}).
\end{equation}
Restricting these forms to $M^n$, we have [8]
\begin{equation}
\omega^{\alpha}=0,
\end{equation}
$$
\omega^{\alpha}_i=\sum_{j}h^{\alpha}_{ij}\, \omega^j,\quad
h=\sum_{\alpha,i,j}h^{\alpha}_{ij}\,\omega^i\otimes\omega^j\otimes e_{\alpha},$$
\begin{equation}
h^{i^*}_{jk}=h^{j^*}_{ik}=h^{k^*}_{ij},
\end{equation}
\begin{equation}
\left\{\begin{array}{l}\mathrm{d}\omega^i=\sum\limits_{j}\omega^j\wedge\omega^i_j,\\
\mathrm{d}\omega^i_j=\sum\limits_{k}\omega^{k}_j\wedge\omega^i_k+\frac{1}{2}
\sum\limits_{k,l}R_{ijkl}\,\omega^k\wedge\omega^l,\\
\end{array}\right.
\end{equation}
\begin{equation}
R_{ijkl}=K_{ijkl}+\sum_{\alpha}(h^{\alpha}_{ik}h^{\alpha}_{jl}
-h^{\alpha}_{il}h^{\alpha}_{jk}),
\end{equation}
\begin{equation}
\mathrm{d}\omega^{\alpha}_{\beta}=\sum_{\gamma}\omega^{\gamma}_{\beta}\wedge
\omega^{\alpha}_{\gamma}+\frac{1}{2}\sum_{k,l}R_{\alpha\beta kl}\,\omega^k\wedge
\omega^l,
\end{equation}
\begin{equation}
R_{\alpha\beta ij}=K_{\alpha\beta ij}+\sum_{k}(h^{\alpha}_{ik}h^{\beta}_{kj}
-h^{\alpha}_{jk}h^{\beta}_{ki}),
\end{equation}
where $h$ is the second fundamental form of $M^n$, and $R_{ijkl}$,
$R_{\alpha\beta ij}$ are the components of the Riemannian curvature
tensor $R$ and the normal curvature tensor $R^\perp$, respectively. We call that the normal connection is flat if $R^\perp=0$.
Let $\zeta$ be the mean curvature vector of $M^n$, i.e.,
$$
\zeta=\frac{1}{n}\sum_{\alpha,j}h^{\alpha}_{jj}e_{\alpha}.
$$
We call $|\zeta|$ the mean curvature of $M^n$, and denote it by
$H$. From the equation of  Gauss (2.10), we have [8]
\begin{equation}
\rho=n(n-1)\frac{c}{4}+n^2H^2-S,
\end{equation}
where $\rho$ is the scalar curvature of $M^n$, and
$S=\sum_{\alpha,i,j}(h^{\alpha}_{ij})^2$. Define the first and the second
covariant derivatives of $h^{\alpha}_{ij}$ as following
\begin{equation}
\sum_{k}h^{\alpha}_{ijk}\,\omega^k=\mathrm{d}h^{\alpha}_{ij}+\sum_{\beta}h^{\beta}_{ij}\,
\omega^{\alpha}_{\beta}-\sum_{l}h^{\alpha}_{lj}\,\omega^l_i-\sum_{l}h^{\alpha}_{il}\,
\omega^l_j,
\end{equation}
\begin{equation}
\sum_{l}h^{\alpha}_{ijkl}\,\omega^l=\mathrm{d}h^{\alpha}_{ijk}+\sum_{\beta}h^{\beta}_{ijk}\,
\omega^{\alpha}_{\beta}-\sum_{l}h^{\alpha}_{ljk}\,\omega^l_i-\sum_{l}h^{\alpha}_{ilk}\,
\omega^l_j-\sum_{l}h^{\alpha}_{ijl}\,
\omega^l_k,
\end{equation}
then [8]
\begin{equation}
h^{\alpha}_{ijk}=h^{\alpha}_{ikj},
\end{equation}
\begin{equation}
h^{\alpha}_{ijkl}-h^{\alpha}_{ijlk}=\sum_{m}(h^{\alpha}_{mi}R_{mjkl}
+h^{\alpha}_{mj}R_{mikl})-\sum_{\beta}h^{\beta}_{ij}R_{\alpha\beta kl}.
\end{equation}
From (2.16), (2.17), the Laplacian of $h^{\alpha}_{ij}$ is
\begin{equation}
  \Delta h^{\alpha}_{ij}=\sum\limits_{k}h^{\alpha}_{kkij}+\sum\limits_{k,m}
  (h^{\alpha}_{im}R_{mkjk}+h^{\alpha}_{km}R_{mijk})-\sum\limits_{\beta,k}h^{\beta}_{ki}
  R_{\alpha\beta jk}.
\end{equation}
From (2.6), we have
\begin{lem}
Let $M^n$ be a totally real submanifold of $\tilde {M}^{n+p}(c)$, then\\
(1) $K_{i^*j^*kl}=K_{ijkl}=\frac{c}{4}(\delta_{ik}\delta_{jl}-\delta_{il}\delta_{jk})$;\\
(2) $K_{\lambda A i j}=0,\ \ K_{\lambda^*Aij}=0;$\\
(3) $K_{\alpha i j k}=0,\ \ K_{\alpha\lambda j k}=0.$
\end{lem}

\section{\bf Position of umbilical normal vectors in the normal bundle}
\vskip 0.4 true cm

Now we assume that $M^n$ is a totally real submanifold with flat normal connection in $\tilde {M}^{n+p}(c)\ (c\not=0)$. For a point $x\in M^n$, $J(T_xM^n)$ is a subspace of the normal space $T^{\perp}_xM^n$, we denote it by $V_x$. Then $T^{\perp}_xM^n$ can be deconposed into the orthonormal sum $T^{\perp}_xM^n=V_x\oplus V^{\perp}_x$, where $V^{\perp}_x$ is the orthonormal complement of $V_x$ in $T^{\perp}_xM^n$. We can characterize the umbilical normal vectors by their position in the normal bundle as the following proposition.

\begin{prop}
  Let $M^n$ be a totally real submanifold with flat normal connection in $\tilde {M}^{n+p}(c)\ (c\not=0)$, $x\in M^n$. Then a nonzero normal vector $\xi$ is umbilical if and only if $\xi\in V^{\perp}_x$ .
\end{prop}

\begin{proof}
  Assume that $\xi$ is a nonzero normal vector in $V^{\perp}_x$, we prove that the shape operator $A_{\frac{\xi}{|\xi|}}$ has exactly one eigenvalue. Choose the orthonormal frame (2.1) at the point $x$, such that $e_1,\cdots,e_n$ are eigenvectors of $A_{\frac{\xi}{|\xi|}}$ and $e_{n+1}=\frac{\xi}{|\xi|}$. Without loss of generality, we may assume that the eigenvalues of $A_{\frac{\xi}{|\xi|}}$ are given by
  \begin{equation*}
    \lambda_1=\cdots=\lambda_{i_1}=\rho_1, \ \lambda_{i_1+1}=\cdots=\lambda_{i_1+i_2}=\rho_2, \cdots,
  \end{equation*}
  \begin{equation*}
    \lambda_{i_1+\cdots+i_{k-1}+1}=\cdots=\lambda_{n}=\rho_k.
  \end{equation*}
  If $\rho_1,\cdots,\rho_k\ (k\geq 2)$ are all distinct, we will get a contradiction. Put
  \begin{equation*}
    [\rho_1]=\{1,\cdots,i_1\},\ [\rho_2]=\{i_1+1,\cdots,i_1+i_2\},\cdots,
  \end{equation*}
  \begin{equation*}
    [\rho_k]=\{i_1+\cdots+i_{k-1}+1,\cdots,n\}.
  \end{equation*}
  With respect to the frame chosen above, the matrix $(h^{n+1}_{ij})$ are given by
  \begin{equation*}
    h^{n+1}_{ij}=\rho_s\delta_{ij},\ \ i,j\in[\rho_s],\ s=1,\cdots,k,
  \end{equation*}
  \begin{equation*}
    h^{n+1}_{ij}=0,\ \ i\in[\rho_s],\ j\in[\rho_t], s\not=t.
  \end{equation*}
  Since the normal connection is flat, by using the equation of Ricci (2.12) and Lemma 2.1, we have
  \begin{align*}
    0 & = R_{n+1\alpha ij}=K_{n+1\alpha ij}+\sum_l(h^{n+1}_{il}h^{\alpha}_{lj}-h^{\alpha}_{il}h^{n+1}_{lj})\\
    & =\sum_l(h^{n+1}_{il}h^{\alpha}_{lj}-h^{\alpha}_{il}h^{n+1}_{lj}).
  \end{align*}
  By setting $i\in [\rho_s],\ j\in [\rho_t],\ s\not=t$, the above equation becomes
  \begin{equation*}
    (\rho_s-\rho_t)h^{\alpha}_{ij}=0,
  \end{equation*}
  which implies that for all $\alpha$,
  \begin{equation}
    h^{\alpha}_{ij}=0,\ \ i\in[\rho_s],\ j\in[\rho_t],\ s\not=t.
  \end{equation}
  By using (2.12) and Lemma 2.1 again, we also have
  \begin{align*}
    0 & = R_{i^*j^*ij}=K_{i^*j^*ij}+\sum_l(h^{i^*}_{il}h^{j^*}_{lj}-h^{j^*}_{il}h^{i^*}_{lj})\\
    & =\frac{c}{4}+\sum_l(h^{i^*}_{il}h^{j^*}_{lj}-h^{j^*}_{il}h^{i^*}_{lj}).
  \end{align*}
  By setting $i\in [\rho_s],\ j\in [\rho_t],\ s\not=t$ in the above equation and considering (3.1), we get $c=0$, which contradicts the fact $\tilde {M}^{n+p}(c)$ is non-flat.

  Conversely, let $\xi$ be an umbilical nonzero normal vector, it can be decomposed into the sum $\xi=\xi_1+\xi_2$, where $\xi_1\in V_x,\ \xi_2\in V^{\perp}_x$. Choose the orthonormal frame (2.1) at the point $x$, such that
  \begin{equation*}
    \xi_1=|\xi_1|e_{1^*},\ \xi_2=|\xi_2|e_{n+1}.
  \end{equation*}
 Noting that $\xi$ is umbilical, we have
 \begin{equation}
   |\xi_1|h^{1^*}_{ij}+|\xi_2|h^{n+1}_{ij}=\langle\xi,\zeta\rangle\delta_{ij},
 \end{equation}
 where $\langle,\rangle$ denotes the inner product in $T^{\perp}_xM^n$. By using the equation of Ricci (2.12), Lemma 2.1 and considering the fact that the normal connection is flat, we obtain
 \begin{equation}
   \sum_l(h^{n+1}_{il}h^{2^*}_{lj}-h^{2^*}_{il}h^{n+1}_{lj})=0,
 \end{equation}
 \begin{equation}
   \sum_l(h^{1^*}_{il}h^{2^*}_{lj}-h^{2^*}_{il}h^{1^*}_{lj})=\frac{c}{4}(\delta_{1j}\delta_{2i}-\delta_{1i}\delta_{2j}),
 \end{equation}
 which, together with (3.2), imply that
 \begin{align*}
   0 & = \sum_l(|\xi_1|h^{1^*}_{il}+|\xi_2|h^{n+1}_{il})h^{2^*}_{lj}-\sum_lh^{2^*}_{il}(|\xi_1|h^{1^*}_{lj}+|\xi_2|h^{n+1}_{lj})\\
     & = |\xi_1|\sum_l(h^{1^*}_{il}h^{2^*}_{lj}-h^{2^*}_{il}h^{1^*}_{lj})+|\xi_2|\sum_l(h^{n+1}_{il}h^{2^*}_{lj}-h^{2^*}_{il}h^{n+1}_{lj})\\
     & = \frac{c}{4}|\xi_1|(\delta_{1j}\delta_{2i}-\delta_{1i}\delta_{2j}).
 \end{align*}
 By setting $i=2,j=1$ in the above equation and noting that $c\not=0$, we get $|\xi_1|=0$, hence $\xi=\xi_2\in V^{\perp}_x$.
\end{proof}

The above proposition can be verified by a classical example which appeared in \cite{YK2,LMK} as follows.

\begin{example}
   Let $\mathbb{C}P^{n+p}$ be a complex projective space of constant holomorphic sectional curvature 4 and of complex dimension $n+p$. Let $S^m(r)$ be the Euclidean sphere of dimension $m$ and radius $r$. Denote $S^m(1)$ just by $S^m$. It is well known that \cite{CCK} $S^1(\frac{1}{\sqrt{3}})\times S^1(\frac{1}{\sqrt{3}})\times S^1(\frac{1}{\sqrt{3}})$ can be naturally immersed in $S^5$ as a minimal submanifold. Then, through the Hopf fiberation $\pi:S^5\to \mathbb{C}P^2$, one can get a minimal totally real submanifold $T^2:=\pi\big(S^1(\frac{1}{\sqrt{3}})\times S^1(\frac{1}{\sqrt{3}})\times S^1(\frac{1}{\sqrt{3}})\big)$ in $\mathbb{C}P^2$ in $\mathbb{C}P^{2+p}$ (see [10]). With respect to some suitable moving frames $\{e_A\}$ of the form (2.1), \cite{YK2} and \cite{LMK} calculated that
   \begin{equation}
   (h^{1*}_{ij})=\left(
   \begin{array}{cc}
    0 & -\frac{1}{\sqrt{2}}\\
    -\frac{1}{\sqrt{2}} & 0\\
   \end{array}
   \right),\ \
    (h^{2*}_{ij})=\left(
   \begin{array}{cc}
    \frac{1}{\sqrt{2}} & 0\\
    0 & -\frac{1}{\sqrt{2}}\\
   \end{array}
   \right),
   \end{equation}
   \begin{equation}
     h^{\alpha}_{ij}=0,\ \alpha\not=1^*,2^*,\ i,j=1,2.
   \end{equation}
   By using (3.5), (3.6), the equation of Ricci (2.12) and Lemma 2.1, one can easily find that the normal connection of $T^2$ in $\mathbb{C}P^{2+p}$ is flat. Also (3.5) and (3.6) implies that for any point $x\in T^2$, the normal vectors in $V_x$ are not umbilical, while the normal vectors in $V^{\bot}_x$ are all umbilical (in fact, in this minimal case, they are all geodesic).
\end{example}

From Proposition 3.2 one can easily get the following corollary.

\begin{cor}
  In non-flat complex space forms, there exist no totally umbilical totally real submanifolds with flat normal connections.
\end{cor}

Combining Corollary 3.2 and a classification result of totally umbilical submanifolds in non-flat complex space forms (see Theorem 1 of \cite{C10}), we have the following

\begin{cor}
  Let $M^n$ be a totally umbilical submanifold in a non-flat complex space form $\tilde{M}^{n+p}(c)$. If the normal connection of $M^n$ is flat, then $M^n$ must be a complex space form immersed holomorphically in $\tilde{M}^{n+p}(c)$ as a totally geodesic submanifold.
\end{cor}


\section{\bf Proof of the theorem}
\vskip 0.4 true cm

Now we assume that $M^n$ is a pseudo-umbilical totally real submanifold with flat normal connection in $\tilde {M}^{n+p}(c)\ (c\not=0)$. If $M^n$ is not minimal, then there exists a point $x_0\in M^n$, such that $H(x_0)\not=0$. From Proposition 3.1 we may choose moving frame (2.1) such that $\zeta=He_{n+1}$ around the point $x_0$, so
\begin{equation}
  \sum_jh^{n+1}_{jj}=nH,\ \ \sum_jh^{\alpha}_{jj}=0\ (\alpha\not=n+1),
\end{equation}
from which, combined with Proposition 3.1, we know that
\begin{equation}
  h^{n+1}_{ij}=H\delta_{ij},\ \ h^{\alpha}_{ij}=0\ (\alpha\not=1^*,\cdots,n^*,n+1),
\end{equation}
which are equivalent to
\begin{equation}
  \omega^{n+1}_i=H\omega^i,\ \ \omega^{\alpha}_{i}=0,\ (\alpha\not=1^*,\cdots,n^*,n+1).
\end{equation}

For convenience, we denote that $\sigma=S-nH^2$. The proof of the theorem is based on the calculation of the Laplacian of $\sigma$. From (2.18), and noting that the normal connection is flat, we have
\begin{align}
  \frac{1}{2}\Delta\sigma  = & \sum_{i,j,k,\alpha}(h^{\alpha}_{ijk})^2+\sum_{i,j,k,\alpha}h^{\alpha}_{ij}h^{\alpha}_{kkij}-\frac{n}{2}\Delta H^2\notag\\
& +\sum_{i,j,k,m,\alpha}h^{\alpha}_{ij}(h^{\alpha}_{im}R_{mkjk}+h^{\alpha}_{km}R_{mijk}).
\end{align}
Write
\begin{align*}
  & {\mathrm{\uppercase\expandafter{\romannumeral1}}}=\sum_{i,j,k,\alpha}(h^{\alpha}_{ijk})^2,\\
  & {\mathrm{\uppercase\expandafter{\romannumeral2}}}=\sum_{i,j,k,\alpha}h^{\alpha}_{ij}h^{\alpha}_{kkij}-\frac{n}{2}\Delta H^2,\\
  & {\mathrm{\uppercase\expandafter{\romannumeral3}}}=\sum_{i,j,k,m,\alpha}h^{\alpha}_{ij}(h^{\alpha}_{im}R_{mkjk}+h^{\alpha}_{km}R_{mijk}).
\end{align*}
We will calculate \uppercase\expandafter{\romannumeral1}, \uppercase\expandafter{\romannumeral2} and \uppercase\expandafter{\romannumeral3} in turn.

From (2.14) and (4.1) one can easily get
\begin{equation}
  \sum_i h^{n+1}_{iik}=nH_k,\ \ \sum_i h^{\alpha}_{iik}=nH\omega^{\alpha}_{n+1}(e_k)\ (\alpha\not=n+1),
\end{equation}
where $H_k=e_k(H)$. Furthermore, by using (2.5) and (4.2), we deduce that
\begin{equation}
  \sum_ih^{m^*}_{iik}=nH\omega^{m^*}_{n+1}(e_k)=nH\omega^{(n+1)^*}_{m}(e_k)=nHh^{(n+1)^*}_{mk}=0.
\end{equation}
\vspace{0.4em}

\begin{lem}
Let $\mathrm{grad}\ H$ denote the gradient of $H$, then
\begin{equation*}
\mathrm{\uppercase\expandafter{\romannumeral1}}=\sum\limits_{i,j,k,m}(h^{m^*}_{ijk})^2+H^2\sigma+\frac{1}{n}
\sum\limits_{\alpha\not=n+1}\sum\limits_{k}\left(\sum\limits_{i}h^{\alpha}_{iik}\right)^2
+n|\mathrm{grad}\ H|^2.
\end{equation*}
\end{lem}
\vspace{0.4em}

\begin{proof}
  From (2.14) and (4.2), we have
  \begin{equation*}
    h^{\alpha}_{ijk}=e_k(h^{\alpha}_{ij})+\sum_{m}h^{m^*}_{ij}\omega^{\alpha}_{m^*}(e_k)+H\delta_{ij}\omega^{\alpha}_{n+1}(e_k)-\sum_{l}h^{\alpha}_{lj}\omega^{l}_{i}(e_k)-\sum_{l}h^{\alpha}_{il}\omega^{l}_{j}(e_k).
  \end{equation*}
  Applying (2.5) and (4.2) to the above formula, it follows that
  \begin{equation*}
      h^{n+1}_{ijk} = H_k\delta_{ij}+\sum_mh^{m^*}_{ij}\omega^{n+1}_{m^*}(e_k)
      =H_k\delta_{ij}-\sum_mh^{m^*}_{ij}\omega^{(n+1)^*}_{m}(e_k)=H_k\delta_{ij},
 \end{equation*}
 \begin{equation*}
   h^{(n+1)^*}_{ijk} = Hh^{k^*}_{ij}+H\delta_{ij}\omega^{(n+1)^*}_{n+1}(e_k),
 \end{equation*}
 \begin{equation*}
   h^{\alpha}_{ijk} = H\delta_{ij}\omega^{\alpha}_{n+1}(e_k),\ (\alpha\not=n+1,1^*,\cdots,(n+1)^*).
 \end{equation*}

 From which, together with (4.2) and (4.5), we may calculate $\mathrm{\uppercase\expandafter{\romannumeral1}}$ as follows.
 \begin{align*}
   \mathrm{\uppercase\expandafter{\romannumeral1}} & = \sum_{i,j,k,m}(h^{m^*}_{ijk})^2+nH^2\sum_{\alpha\not=n+1,1^*,\cdots,(n+1)^*}\sum_{k}\big(
    \omega^{\alpha}_{n+1}(e_k)\big)^2+n|\mathrm{grad}\ H|^2\\
    & \ \ \ \ +H^2\sum_{i,j,k}\left(h^{k^*}_{ij}+\delta_{ij}\omega^{(n+1)^*}_{n+1}(e_k)\right)^2\\
    & =\sum_{i,j,k,m}(h^{m^*}_{ijk})^2+n|\mathrm{grad}\ H|^2+nH^2\sum_{\alpha\not=n+1,1^*,\cdots,(n+1)^*}\sum_{k}
    \big(\omega^{\alpha}_{n+1}(e_k)\big)^2+H^2\sigma\\
    & \ \ \ \ +2H^2\sum_{i,k}h^{k^*}_{ii}\omega^{{(n+1)}^*}_{n+1}(e_k)+nH^2\sum_{k} \big(\omega^{{(n+1)}^*}_{n+1}(e_k)\big)^2\\
    & = \sum_{i,j,k,m}(h^{m^*}_{ijk})^2+n|\mathrm{grad}\ H|^2+H^2\sigma+nH^2\sum_{\alpha\not=n+1,1^*,\cdots,n^*}\sum_{k}\big(\omega^{\alpha}_{n+1}(e_k)\big)^2\\
    & = \sum_{i,j,k,m}(h^{m^*}_{ijk})^2+n|\mathrm{grad}\ H|^2+H^2\sigma+\frac{1}{n}\sum_{\alpha\not=n+1,1^*,\cdots,n^*}\sum_{k}\left(\sum_{i}h^{\alpha}_{iik}\right)^2.
\end{align*}
This completes the proof.
\end{proof}

\begin{lem}
$\mathrm{\uppercase\expandafter{\romannumeral2}}=-\frac{1}{n}\sum\limits_{\alpha\not=n+1}\sum\limits_{k}\left(\sum\limits_{i}
h^{\alpha}_{iik}\right)^2-n|\mathrm{grad}\ H|^2.$
\end{lem}

\begin{proof}
  From (4.2),
  $$
\sum_{\alpha,i,j,k}h^{\alpha}_{ij}h^{\alpha}_{kkij}=\sum_{m,i,j,k}
h^{m^*}_{ij}h^{m^*}_{kkij}+\sum_{i,j,k}h^{n+1}_{ij}h^{n+1}_{kkij}.$$
We compute each term of the right hand side of the above formula as follows. During the procedure, we have used the formulas (2.5), (2.8), (4.2) and (4.6).
\begin{align*}
\sum_{m,i,j,k}h^{m^*}_{ij}h^{m^*}_{kkij}
     = & \sum_{m,i,j,k}h^{m^*}_{ij}\left(e_j(h^{m^*}_{kki})+\sum_{\beta}h^{\beta}_{kki}
        \omega^{m^*}_{\beta}(e_j) -\sum_{l}h^{m^*}_{lki}\omega^{l}_{k}(e_j)\right.\\
       & \left.-\sum_{l}h^{m^*}_{kli}\omega
        ^{l}_{k}(e_j)-\sum_{l}h^{m^*}_{kkl}\omega^{l}_{i}(e_j)\right)\\
     = &\sum_{m,i,j,k}\sum_{\beta}h^{m^*}_{ij}h^{\beta}_{kki}\omega^{m^*}_{\beta}(e_j)\\
     = & \sum_{m,i,j,k,l}h^{m^*}_{ij}h^{l^*}_{kki}\omega^{m^*}_{l^*}(e_j)+\sum_{m,i,j,k}h^{m^*}_{ij}h^{n+1}_{kki}\omega^{m^*}_{n+1}(e_j)\\
     & + \sum_{m,i,j,k}\sum_{\lambda\not=n+1}h^{m^*}_{ij}h^{\lambda}_{kki}\omega^{m^*}_{\lambda}(e_j)+\sum_{m,i,j,k}\sum_{\lambda}h^{m^*}_{ij}h^{\lambda^*}_{kki}\omega^{m^*}_{\lambda^*}(e_j)\\
      = & -\sum_{m,i,j,k}\sum_{\lambda}h^{m^*}_{ij}h^{\lambda^*}_{kki}\omega^{\lambda}_{m}(e_j)\\
     = & -H\sum_{m,i,j,k}h^{m^*}_{ij}\delta_{mj}h^{(n+1)^*}_{kki}\\
    = & 0,
\end{align*}
\begin{align*}
  \sum_{i,j,k}h^{n+1}_{ij}h^{n+1}_{kkij} & = \sum_{i,j,k}H\delta_{ij}\Big(e_j(h^{n+1}_{kki})+\sum_{\beta}h^{\beta}_{kki}\omega^{n+1}_{\beta}(e_j)-\sum_lh^{n+1}_{lki}\omega^{l}_{k}(e_j)\\
  & \ \ \ \  - \sum_{l}h^{n+1}_{kli}\omega^{l}_{k}(e_j)-\sum_{l}h^{n+1}_{kkl}\omega^{l}_{i}(e_j)\Big)\\
  = & \sum_{i,j}H\delta_{ij}\left(e_j(nH_i)-nH_l\omega^{l}_{i}(e_j)-\frac{1}{nH}\sum_{\beta\not=n+1}\left(\sum_kh^{\beta}_{kki}\right)\left(\sum_lh^{\beta}_{llj}\right)\right)\\
  = & \sum_{i,j}H\delta_{ij}\left(nH_{ij}-\frac{1}{nH}\sum_{\beta\not=n+1}\left(\sum_kh^{\beta}_{kki}\right)\left(\sum_lh^{\beta}_{llj}\right)\right)\\
  = & nH\Delta H-\frac{1}{n}\sum_{\beta\not=n+1}\sum_i\left(\sum_kh^{\beta}_{kki}\right)^2.
  \end{align*}
  Noting that
  \begin{equation*}
    \Delta H^2=2H\Delta H+2|\mathrm{grad}\ H|^2,
  \end{equation*}
 hence
 \begin{equation*}
   \mathrm{\uppercase\expandafter{\romannumeral2}}=-\frac{1}{n}\sum\limits_{\alpha\not=n+1}\sum\limits_{i}\left(\sum\limits_{k}
h^{\alpha}_{kki}\right)^2-n|\mathrm{grad}\ H|^2.
 \end{equation*}
\end{proof}

\begin{lem}
  $\mathrm{\uppercase\expandafter{\romannumeral3}}=nH^2\sigma.$
\end{lem}

\begin{proof}
  Applying Lemma 2.1, (4.2), (2.8) and (2.12) to the equation of Gauss (2.10), noting that the normal connection is flat, we have
  \begin{align}
    R_{ijkl} & = K_{ijkl}+\sum_{\alpha}(h^{\alpha}_{ik}h^{\alpha}_{jl}
             -h^{\alpha}_{il}h^{\alpha}_{jk})\notag\\
         & = K_{i^*j^*kl}+\sum_{m}(h^{m^*}_{ik}h^{m^*}_{jl}-h^{m^*}_{il}h^{m^*}_{jk})
             +H^2(\delta_{ik}\delta_{jl}-\delta_{il}\delta_{jk})\notag\\
         & = K_{i^* j^*kl}+\sum_{m}(h^{i^*}_{km}h^{j^*}_{ml}-h^{i^*}_{lm}h^{j^*}_{mk})
             +H^2(\delta_{ik}\delta_{jl}-\delta_{il}\delta_{jk})\notag\\
         & = R_{i^*j^*kl}+H^2(\delta_{ik}\delta_{jl}-\delta_{il}\delta_{jk})\notag\\
         & = H^2(\delta_{ik}\delta_{jl}-\delta_{il}\delta_{jk}).
  \end{align}
  Then Lemma 4.3 follows by a direct calculation.
\end{proof}

\begin{proof}[Proof of Theorem 1.1]\ Assume $M^n$ is not minimal, then there exists a point $x_0\in M^n$, such that $H(x_0)\not=0$. Applying Lemma 4.1, Lemma 4.2 and Lemma 4.3 to (4.4), we have
\begin{equation}
  \frac{1}{2}\Delta\sigma=\sum_{m,i,j,k}(h^{m^*}_{ijk})^2+(n+1)H^2\sigma.
\end{equation}
On the other hand, (4.7) implies that the scalar curvature of $M^n$ is given by
\begin{equation*}
  \rho=n(n-1)H^2,
\end{equation*}
substituted into (2.13), we see that $\sigma=\frac{c}{4}n(n-1)$. Thus (4.8) becomes
\begin{equation*}
  0=\sum_{m,i,j,k}(h^{m^*}_{ijk})^2+\frac{c}{4}n(n^2-1)H^2.
\end{equation*}
Hence $H=0$ around the point $x_0$, which contradicts the assumption $H(x_0)\not=0$. Therefore, $M^n$ must be minimal. This completes the proof.
\end{proof}


\vskip 0.5 true cm

{\bf Acknowledgements}\ \ This work is supported by the Foundation for Excellent Young Talents of Higher Education of Anhui Province (Grant No.\,2011SQRL021ZD).


\bigskip
\bigskip

\end{document}